\documentclass[12pt,leqno]{amsart}
\usepackage{graphicx}
\usepackage{amsfonts,amsthm,amsmath,mathtools,amssymb,comment,xcolor,enumitem}
\usepackage[unicode]{hyperref}
\usepackage{cleveref}
\usepackage{ascmac}
\usepackage{mathtools}
\usepackage{soul}

\theoremstyle{plain}
\newtheorem{thm}{Theorem}[section]

\newtheorem{lem}{Lemma}[section]
\newtheorem{cor}{Corollary}[section]

\theoremstyle{definition}
\newtheorem{df}{Definition}[section]
\newtheorem{rem}{Remark}[section]

\newcommand{\FF}{\mathbb{F}}

\newcommand{\allone}{\mathbf{1}}

\begin{document}

\title{Average hitting times in some $f$-equitable graphs}

\author{Yusaku Nishimura}
\address{School of Fundamental Science and Engineering, Waseda University, Tokyo 169-8555, Japan}
\email{n2357y@ruri.waseda.jp}
\date{}

\begin{abstract}
  It is known that the average hitting times of simple random walks from any vertex to any other vertex in distance-regular graphs are determined by their intersection array.
  In this paper, we introduce a new graph classification called $f$-equitable, utilizing both the equitable partition and the function $f$, which represents a generalization of distance-regular graphs. 
  We determine the average hitting times from any vertex to any other vertex in $f$-equitable graphs by using their parameter referred to as the quotient matrix. 
  Furthermore, we prove that there is some function $f$ such that the Cartesian product of two strongly regular graphs is $f$-equitable.
  We then calculate the quotient matrix for these graphs and determine the average hitting times from any vertex to any other vertex in these graphs.
  In the same manner, we determine the average hitting times on some generalized Paley graphs.
\end{abstract}

\maketitle

	

\section{Introduction}

A \textit{random walk} on a graph is a stochastic process that moves randomly from vertex to vertex along edges. 
Specifically, when each vertex is chosen with equal probability during a random walk, it is referred to as a \textit{simple random walk}.
For any given vertices $u$ and $v$ in a graph, the \textit{hitting time} from vertex $u$ to vertex $v$ is defined as the number of steps required to reach vertex $v$ 
for the first time in a simple random walk starting from vertex $u$.
The expected value of the hitting times from vertex $u$ to vertex $v$ is called the \textit{average hitting time} from $u$ to $v$ and is denoted as $H(G;(v,u))$. We define $H(G;(v,v))=0$.

It is known that we can compute the average hitting time between any two vertices easily in certain types of graphs. 
A notable example is distance-regular graphs, which belong to a class of graphs with high symmetry.
Biggs~\cite{biggs1993}, Devroye and Sbihi~\cite{Devroye1990}, and Van Slijpe~\cite{Slijpe} independently demonstrated that 
the average hitting times between any two vertices in a distance-regular graph can be computed solely using the intersection array of the graph, 
which is a set of parameters defining the distance-regular graph. 
In this paper, we propose a new graph classification denoted as $f$-equitable.
This classification can be regarded as a generalization of distance-regular graphs and also as a generalization of association schemes.

The detailed definition of $f$-equitable is given in \Cref{sec:stab}.

In the following, we denote $\vec{\allone}$ as the all-one vector and $I$ as the identity matrix.
We determine the average hitting time from any vertex to any vertex in an $f$-equitable graph by using its parameter, denoted as the quotient matrix.
\begin{thm}\label{thm:equi}
  Let $G$ be an $f$-equitable graph with the quotient matrix $Q$ and valency $k$, 
  where the range of the function $f$ is $\{x_0, x_1, \ldots, x_r\}$, and $Q$ is indexed in the order of $\{x_0, x_1, \ldots, x_r\}$. 
  Let $h_{x_0}=0$, and we define $h_{x_i}$ as follows:
  \begin{align*}
    \vec{H}&=(h_{x_1},h_{x_2},\ldots,h_{x_r})\\
    &=-\vec{\allone}\left(\frac{1}{k}Q^--I\right)^{-1},  
  \end{align*}
  where $Q^-$ is the submatrix obtained by removing the $x_0$th row and the $x_0$th column from $Q$. 
  Then, for any vertices $u$ and $v$ in $G$,
  \[
  H(G;(v,u))=h_{f(v,u)}.
  \]
\end{thm}

\Cref{thm:equi} implies that if $G$ is $f$-equitable, then the average hitting time between any two vertices in $G$ is determined using only the quotient matrix and the function $f$.
Using this method, we can calculate the average hitting time in the Cartesian product of any two strongly regular graphs 
and specific generalized Paley graphs using only these parameters, which was previously unknown.
In the same manner, the detailed definitions of these graphs and parameters are provided in \Cref{sec:drg,sec:srg,sec:genPaley}.
Note that these graphs are obtained by selecting multiple relations from an association scheme and treating them as the edge set. 
Generally, the graph obtained through such operations is an $f$-equitable, where $f$ is a function that returns the index of relation containing a pair of vertices.
Therefore, \Cref{thm:equi} can be applied to all graphs generated by such operations, which is a much broader class than distance regular graphs.

This paper is organized as follows.
In \Cref{sec:equitable}, we provide the definition of an equitable partition and discuss some of its properties.
In \Cref{sec:stab}, we define $f$-equitable and present the proof for \Cref{thm:equi}.
In \Cref{sec:drg}, we calculate the average hitting time in distance-regular graphs by applying \Cref{thm:equi}, and confirm that the results are consistent with those previously obtained.
In \Cref{sec:srg,sec:genPaley}, we calculate the average hitting time in the Cartesian product of two strongly regular graphs and specific generalized Paley graphs, respectively. 
The actual calculations were performed using Mathematica~\cite{Mathematica}.

\section{Definition of an equitable partition and some properties}\label{sec:equitable}

In the following, $N_G(v)$ denotes the neighbor set of vertex $v$ in graph $G$, and $\deg_G(v)$ denotes the degree of $v$.
When $G$ is understood, we omit it and write $N(v)$ and $\deg(v)$, respectively.  
We also denote $u\sim v$ if vertices $u$ and $v$ are adjacent.

\subsection*{Definition and some properties of an equitable partition }
Let $G$ be a graph and let $P=\{V_0,V_1,\ldots,V_{r}\}$ be a partition of the vertex set of $G$.

\begin{df}[Equitable partition~\cite{10.1007/BFb0066438}]\label{def:equitable}
  The partition $P$ is called an \textit{equitable partition} if, for all integers $i$ and $j$ such that $0\leq i,j\leq r$ and for all vertices $u$ and $v$ in $V_j$, 
  $|N(u)\cap V_i|=|N(v)\cap V_i|$ holds.
\end{df}

\begin{df}[Quotient matrix]
  Let $P$ be an equitable partition of $G$.
  A matrix $q_{ij}=|N(v)\cap V_i|$, where $v$ is for any vertex in $V_j$, is called a \textit{quotient matrix} of $P$.
\end{df}

We can state some properties of the quotient matrix of an equitable partition.

\begin{lem}\label{eq:deg}
  Let $G$ be a graph, and let $P=\{V_0,V_1,\ldots,V_{r}\}$ be an equitable partition of $G$ with quotient matrix $q_{ij}$. 
  \begin{equation*}
    v\in V_j \Longrightarrow \deg(v)=\sum_{i=0}^{r}q_{ij}. 
  \end{equation*}
\end{lem}
\begin{proof}
  Define $V$ as the vertex set of $G$.
  \begin{align*}
    \deg(v)&=|N(v)\cap V|\\
    &=\sum_{i=0}^{r}|N(v)\cap V_i|.
  \end{align*}
  From the definition of an equitable partition and $q_{ij}$, 
  \[
    v\in V_j \Longrightarrow \sum_{i=0}^{r}|N(v)\cap V_i|=\sum_{i=0}^{r}q_{ij}.
  \]
\end{proof}

\begin{lem}\label{eq:edge}
  Let $G$ be a graph, and let $P=\{V_0,V_1,\ldots,V_{r}\}$ be an equitable partition of $G$ with quotient matrix $q_{ij}$.
  \begin{equation*}
    |V_i|q_{ji}=|V_j|q_{ij}.
  \end{equation*}  
\end{lem} 
\begin{proof}
  By counting the edges between $V_i$ and $V_j$ in two different ways, we obtain this lemma.
\end{proof}

\subsection*{Random walk on a graph with equitable partition}

Let us consider a simple random walk on graph $G$ with an equitable partition $P=\{V_0,V_1,\ldots,V_{r}\}$.
Let $X_0=x$ be the starting vertex of a simple random walk on $G$, and let $X_t$ be a random variable representing the vertex reached after $t$ steps in the random walk. 
Similarly, let $Y_0=V_{p_0}\ni X_0$ be the starting partition of the simple random walk on $G$, and let $Y_t=V_{p_t}\ni X_t$ be the random variable representing the partition reached after $t$ steps in the random walk. 
Then, $Y_t$ also represents a stochastic process that moves randomly from $V_j$ to $V_i$, where $V_i$ and $V_j$ are elements of $P$.

Since this partition $P$ is equitable, we can compute the transition probabilities using the quotient matrix of $P$.
Let us denote $Pr(Y_t=V_i\mid Y_{t-1}=V_{j})$ as the probability of $Y_t=V_i$ given that $Y_{t-1}=V_{j}$.
Using \Cref{eq:deg},
\begin{equation}\label{lem:part}
  Pr(Y_t=V_i\mid Y_{t-1}=V_{j})=\frac{|N(X_{t-1})\cap V_i|}{\deg(X_{t-1})}=\frac{q_{ij}}{\sum_{l=0}^{r}q_{lj}}.
\end{equation}

We can define the hitting time from $V_j$ to $V_i$, which is similar to the hitting time from vertex to vertex.
For any integers $i$ and $j$ such that $0\leq i,j\leq r$, we define the hitting time from $V_j$ to $V_i$ as the number of steps required 
to reach any vertex $v\in V_i$ for the first time in a simple random walk starting from any vertex $u\in V_j$. 
We also define the average hitting time from $V_j$ to $V_i$ as the expected value of the hitting times from $V_j$ to $V_i$, denoted as $H((G,P);(i,j))$. 
Because \Cref{lem:part} shows that the transition probability $Pr(Y_t=V_i\mid Y_{t-1}=V_{j})$ is independent of the choice of $u\in V_{j}$ and $v\in V_i$, $H((G,P);(i,j))$ is well-defined. 
Similar to $H(G;(v,v))$, we define $H((G,P);(i,i))=0$.

In the following, we denote $T_{G,P}$ as the transition probability matrix from $V_j$ to $V_i$, where $P=\{V_0,\ldots,V_r\}$ and its $ij$-entry as 
\[
  t(G,P)_{ij}=\frac{q_{ij}}{\sum_{k=0}^{r}q_{kj}}.
\] 
When $G$ and $P$ are clear, we omit them and write $T$ and $t_{ij}$.

\section{Definition of $f$-equitable and proof of \Cref{thm:equi}}\label{sec:stab}

In this section, we first provide the definitions of $f$-equitable.
Following that, we give a proof of \Cref{thm:equi}.

First, we define a stabilized equitable partition, which is useful for considering the average hitting time from any vertex to a specific vertex.

\begin{df}[Stabilized equitable partition]
  Let $G$ be a graph and let $P=\{V_0,V_1,\ldots,V_r\}$ be its equitable partition.
  $P$ is called a stabilized equitable partition centered on vertex $o$ if $V_0=\{o\}$.
\end{df}

When $P$ is a stabilized equitable partition centered on vertex $o$, the average hitting time from any vertex to $o$ can be calculated using $H((G,P);(0,i))$.

\begin{lem}\label{thm:main}
  Let $G$ be a graph and $P=\{V_0=\{o\},V_1,\ldots,V_r\}$ denote its stabilized equitable partition centered on vertex $o$.
  For any vertex $v$ in $V_i$,
  \[
    H(G;(o,v))  = H((G,P);(0,i)).
  \]
\end{lem}

\begin{proof}
  From the definition, the average hitting time from $V_i$ to $V_0$ represents the average hitting time from any vertex $u\in V_i$ to $v\in V_0$.
  In this case, since $V_0=\{o\}$, $v\in V_0$ means $v=o$.
  Therefore, $H(G;(o,v))=H((G,P);(0,i))$.
\end{proof}

We can compute $H((G,P);(0,i))$ using the transition probability matrix $t_{ij}$.

\begin{thm}\label{eq:ht}
  Let $T^-$ be the submatrix obtained by removing the $0$th row and the $0$th column from $t_{ij}$.
  Then, 
  \begin{align*}
    \vec{H}&=(H((G,P);(0,1)),H((G,P);(0,2)),\ldots,H((G,P);(0,r)))\\
    &=-\vec{\allone}(T^--I)^{-1}.
  \end{align*}
\end{thm}
\begin{proof}
Similar to the equations formulated for the average hitting time from vertex to vertex, equations can also be formulated for $H((G,P);(0,i))$ as follows:
\begin{align*}
  H((G,P);(0,i))&=\sum_{k=0}^r(1+H((G,P);(0,k)))t_{ki}\\
  &=1+\sum_{k=0}^rH((G,P);(0,k))t_{ki}.  
\end{align*}
When considering this equation for all $i$, we obtain $\vec{H}T^- + \vec{\allone} = \vec{H}$ and, by transforming the equation,
\[
  \vec{H}=-\vec{\allone}(T^--I)^{-1}.
\]
\end{proof}

\Cref{eq:ht} implies that if $P$ is a stabilized equitable partition centered on vertex $o$, then we can compute $H(G;(o,u))$ for any vertex $u$ in $G$ 
using the transition probability matrix, which is calculated by the quotient matrix of $P$.

This method can obviously compute only the average hitting time ending at a specific vertex. 
However, in some graphs with high symmetry, we can compute the average hitting time from any vertex to any vertex using this method. 
We classify such graphs as $f$-equitable graphs, where $f$ is a function that takes two vertices of a graph as arguments.
\begin{df}[$f$-equitable graph]
  Let $G$ be a graph with the vertex set $V$, and let $f:V\times V\mapsto \{x_0,x_1,\ldots,x_r\}$ be a function.
  Define $F_{x_i}(o)=\{v\in V\mid f(o,v)=x_i\}$ as a subset of $V$, and let 
  \[
    P_o=\{F_{x_0}(o),F_{x_1}(o),\ldots,F_{x_r}(o)\}
  \] 
  be the partition of $V$, where $o$ is any vertex in $V$.
  $G$ is called an \textit{$f$-equitable graph} if, for all vertices $o$ in the graph $G$, $P_o$ satisfies the following three conditions:
 \begin{itemize}
  \item $F_{x_0}(o)=\{o\}$;
  \item $P_o$ is equitable;
  \item the quotient matrix of $P_o$ does not depend on the choice of $o$.
 \end{itemize}
\end{df}

Note that the quotient matrix of $P_o=\{F_{x_0}(o),F_{x_1}(o),\ldots,F_{x_r}(o)\}$ is indexed by $x_0,x_1,\ldots,x_r$.
In the following, we denote the quotient matrix of an $f$-equitable graph as the quotient matrix of $P_o$.
Before the proof of \Cref{thm:equi}, we give a simple property of an $f$-equitable graph.
\begin{lem}
  Let $G$ be an $f$-equitable graph with the quotient matrix $Q$.
  Then, $G$ is a regular graph with valency $\sum_{k=0}^{r}q_{x_kx_0}$, 
  where $q_{x_ix_j}$ is the entry of the $x_i$th row and $x_j$th column of $Q$.
\end{lem}
\begin{proof}
  From the definition of $f$-equitable, for all vertices $o$ in $G$, there exists a stabilized equitable partition centered on $o$ and its quotient matrix is $Q$. 
  The sum of the column $x_0$ of $Q$ represents the degree of $o$.
  Since $Q$ does not depend on the choice of $o$, neither does the degree of $o$.
  Thus, the proof is completed.
\end{proof}

Then, we prove \Cref{thm:equi}.

\begin{proof}[Proof of \Cref{thm:equi}]
  Define $F_{x_i}(v)=\{u\in V\mid f(v,u)=x_i\}$ as the subset of $V$ and $P_v=\{F_{x_0}(v),F_{x_1}(v),\ldots,F_{x_r}(v)\}$ as the partition of $G$.
  From the definition of $f$-equitable, $P_v$ is a stabilized equitable partition centered on $v$ and its quotient matrix is $Q$.
  Since the valency of $G$ is $k$, from \Cref{lem:part}
  \[
    T_{G,P_v}=\frac{1}{k}Q.
  \]
  From \Cref{eq:ht},
\begin{align*}
  \vec{H'}&=(H((G,P_v);(x_0,x_1)),H((G,P_v);(x_0,x_2)),\ldots,H((G,P_v);(x_0,x_r)))\\
  &=-\vec{\allone}(T_{G,P_v}^--I)^{-1}\\
  &=-\vec{\allone}\left(\frac{1}{k}Q^--I\right)^{-1}\\
  &=\vec{H}.
\end{align*}
  Therefore, $\vec{H}=\vec{H'}$.
  Especially,
\[
h_{x_i}=  H((G,P_v);(x_0,x_i)),
\]
where $1\leq i\leq r$.
  From \Cref{thm:main}, for any vertex $u$ in $F_{x_i}(v)$,
  \begin{align*}
    H(G;(v,u))=H((G,P_v);(x_0,x_i)).
  \end{align*}
  From the definition of $F_{x_i}(v)$, $u\in F_{x_i}(v)$ if and only if $f(v,u)=x_i$.
  Therefore, for any vertex $u$ not in $F_{x_0}(v)$ in $G$,
  \begin{align*}
    H(G;(v,u))=H((G,P_v);(x_0,f(v,u)))=h_{f(v,u)}.
  \end{align*}
  Given that $F_{x_0}(v)=\{v\}$, $u\in F_{x_0}(v)$ if and only if $u=v$.
  Hence,
  \[
    H(G;(v,u))=\begin{cases}
      0 & f(v,u)=x_0\\
      h_{f(v,u)} &f(v,u)\neq x_0
    \end{cases}.  
  \]
  Since $h_{x_0}=0$, by organizing the conditions,
  \[
    H(G;(u,v))=h_{f(v,u)}.
  \]
\end{proof}

In short, we can compute the average hitting time from any vertex to any vertex in an $f$-equitable graph using the quotient matrix and the function $f$. 
In the following, we denote
\[
  \vec{H}(A)=-\vec{\allone}(A^--I)^{-1},
\]
where $A$ is square matrix indexed by $\{x_0,\ldots,x_r\}$ and $A^-$ is the submatrix obtained by removing the $x_0$th row and the $x_0$th column from $A$.

We will consider specific examples of $f$-equitable graphs for which we can specify the quotient matrix.

\section{Average hitting time on a distance-regular graph}\label{sec:drg}

In this section, we determine the average hitting time of a distance-regular graph using \Cref{thm:equi}.

We define $d_G(u,v)$ as the distance between two vertices $u$ and $v$ in graph $G$, representing the length of the shortest path connecting $u$ and $v$, and define $d_G(v,v)=0$.
We also define $\Delta_G$ as the diameter of $G$, which is the maximum distance $d_G(u,v)$ among all pairs of vertices $u$ and $v$ in $G$.

\begin{df}[Distance-regular graph~\cite{key}]\label{df:drg}
  A $k$-regular graph $G$ is called \textit{distance-regular} if there exist integers $b_i$ and $c_i$, where $0\leq i\leq \Delta_G$, 
  such that for any pair of vertices $u$ and $v$ in the graph with a distance of $i$ between them, 
  there are precisely $c_i$ neighbors of $u$ in the set of vertices at a distance of $i-1$ from $v$ and $b_i$ neighbors of $u$ in the set of vertices at a distance of $i+1$ from $v$. 
  This integer array, denoted as $\{b_0=k,b_1,\ldots,b_{\Delta_G-1};c_1=1,c_2,\ldots,c_{\Delta_G}\}$, is called the \textit{intersection array}. 
\end{df}

In fact, the distance-regular graph is a notable example of a $d_G$-equitable graph. 

\begin{lem}\label{cor:drgEqui}
  $G$ is a distance-regular graph with the intersection array $\{b_0=k,\ldots,b_{\Delta_G-1};c_1,\ldots,c_{\Delta_G}\}$
  if and only if $G$ is a $d_G$-equitable graph with the following quotient matrix $Q$:
  \begin{align*}
    Q=  \begin{pmatrix}
      0 & 1 & 0 & \ldots & 0\\
      k & k-b_1-c_1 &c_2&\ldots &0\\
      0 & b_1 &k-b_2-c_2&\ldots &0\\
      0&     &\vdots& & 0\\
      0&\ldots&\ldots&k-b_{\Delta_G-1}-c_{\Delta_G-1} &c_{\Delta_G}\\
      0&\ldots&\ldots&b_{\Delta_G-1} &k-b_{\Delta_G}-c_{\Delta_G}\\
    \end{pmatrix}.  
  \end{align*}
  \end{lem}
\begin{proof}
Let $G$ be a distance-regular graph with the intersection array $\{b_0=k,\ldots,b_{\Delta_G-1};c_1=1,\ldots,c_{\Delta_G}\}$.
Define $D_i(v)=\{u\in V\mid d_G(v,u)=i\}$ and partition $P_v=\{D_0(v),D_1(v),\ldots,D_{\Delta_G}(v)\}$.
From the definition of distance-regular graphs, for any vertex $v\in G$,
\begin{equation}\label{eq:inter}
  u\in D_j(v)\Longrightarrow |N(u)\cap D_i(v)|=
    \begin{cases}
      c_i &  i = j-1\\
      k-b_i-c_i &  i = j\\
      b_i &  i = j+1\\
      0 & \text{others}\\
    \end{cases}  
\end{equation}
This implies $P_v$ is the equitable partition and its quotient matrix is $Q$ as follows:
\[
Q=
\begin{pmatrix}
  0 & 1 & 0 & \ldots & 0\\
  k & k-b_1-c_1 &c_2&\ldots &0\\
  0 & b_1 &k-b_2-c_2&\ldots &0\\
  0&     &\vdots& & 0\\
  0&\ldots&\ldots&k-b_{\Delta_G-1}-c_{\Delta_G-1} &c_{\Delta_G}\\
  0&\ldots&\ldots&b_{\Delta_G-1} &k-b_{\Delta_G}-c_{\Delta_G}\\
\end{pmatrix}  
\]
Here, we can define any vertex in $G$ as $v$, and $Q$ is independent of the choice of $v$.
Furthermore, from the definition of $d_G$, $D_0(v)=\{v\}$.
Therefore, the distance-regular graph is a $d_G$-equitable graph with the quotient matrix $Q$.

Conversely, let $G$ be a $d_G$-equitable graph with the quotient matrix $Q$.
From the definition of the quotient matrix, we also have \Cref{eq:inter}, 
and this implies $G$ is a distance-regular graph with the intersection array $\{b_0=k,\ldots,b_{\Delta_G-1};c_1,\ldots,c_{\Delta_G}\}$.
\end{proof}

Using \Cref{thm:equi}, we can compute the average hitting times in distance-regular graphs.
\begin{thm}\label{cor:drg}
Let $G$ be a distance-regular graph with intersection array $\{b_0=k,b_1,\ldots,b_{\Delta_G-1};c_1=1,\ldots,c_{\Delta_G}\}$. 
Consider vertices $u$ and $v$ in this distance-regular graph $G$, and let $Q$ be the quotient matrix of $G$ as a $d_G$-equitable graph.
Then,
\[
  H(G;(u,v))=h_{d_G(u,v)},
\]
 where $\vec{H}(\frac{1}{k}Q)=(h_1,h_2,\ldots,h_{\Delta_G})$ and $h_0=0$.
\end{thm}

\begin{proof}
  From \Cref{thm:equi},
  \[
    H(G;(u,v))=h_{d_G(u,v)}. 
  \]
\end{proof}

\Cref{cor:drg} was previously obtained independently by Biggs~\cite{biggs1993}, Devroye and Sbihi~\cite{Devroye1990}, and Van Slijpe~\cite{Slijpe}, 
and they computed $\vec{H}(\frac{1}{k}Q)$.

\begin{thm}[\cite{biggs1993,Devroye1990,Slijpe}]
  \[
  h_i=  \sum_{j=1}^{i}\frac{k}{k_{j-1}b_{j-1}}\sum_{l=j}^{\Delta_G}k_l,
  \]
  where $k_i$ is an integer array defined for $0\leq i\leq \Delta_G$ such that $k_0=1$ and $k_i=\frac{b_{i-1}}{c_i}k_{i-1}$. 
\end{thm}

\section{Average hitting times of Cartesian product of two strongly regular graphs}\label{sec:srg}

In this section, we give a proof that if $G$ and $G'$ are $f$-equitable and $f'$-equitable, respectively, then the Cartesian product of $G$ and $G'$ is $(f,f')$-equitable, 
where $(f,f')$ is defined by $f$ and $f'$. 
Using this result, we calculate the average hitting time from any vertex to another on the Cartesian product of two strongly regular graphs.
Before the discussion, we define the Cartesian product of two graphs.

\begin{df}[Cartesian product of two graphs]
  Let $G$ and $G'$ be graphs with vertices $V$ and $V'$, respectively. 
  The Cartesian product of $G$ and $G'$ is defined by the vertex set $V\times V'$, 
  and the vertices $(u,u')$ and $(v,v')$ in $V\times V'$ are adjacent if 
  \[
    ((u = v) \land (u'\sim v'))\lor ((u \sim v) \land (u' = v')).
  \]
  \end{df}

  In the following, we define $G\times G'$ as the Cartesian product of two graphs $G$ and $G'$.
\subsection*{Cartesian product of two graphs that are $f$-equitable and $f'$-equitable}

In this subsection, we prove the following theorem.

\begin{thm}\label{thm:carteq}
  Let $G$ be an $f$-equitable graph with the quotient matrix $Q$ and vertex set $V$.
  Similarly, let $G'$ be an $f'$-equitable graph with the quotient matrix $Q'$ and vertex set $V'$.
  Then, $G\times G'$ is an $(f,f')$-equitable graph with the quotient matrix \[Q\otimes I_{r'+1}+I_{r+1}\otimes Q',\] 
  where $r+1,r'+1$ is the size of $Q,Q'$, respectively, and $(f,f')$ is a function such that 
  \begin{align*}
    (f,f'):(V\times V')\times(V\times V')  &\mapsto f(V\times V)\times f'(V'\times V')\\
    ((u,u'),(v,v'))&\mapsto(f(u,v),f'(u',v')).
  \end{align*}
\end{thm}

To prove \Cref{thm:carteq}, we establish the following lemma.

\begin{lem}\label{lem:cart}
  Let $G,G'$ be graphs with equitable partitions $P=\{V_0,V_1,\ldots,V_r\}$ and $P'=\{V'_0,\ldots,V'_{r'}\}$, respectively.
  The quotient matrices of these partitions are denoted as $Q$ and $Q'$, respectively. 
  Then, 
  \[P\times P'=\{(V_i,V'_j)\mid 0\leq i\leq r,0\leq j\leq r'\}\]
  is an equitable partition of $G\times G'$.
  The quotient matrix of $P\times P'$ is given by $Q\otimes I_{r'+1}+I_{r+1}\otimes Q'$.
\end{lem}

\begin{proof}
  Using the definition of the Cartesian product, we have the following equation:
  \[
  N_{G\times G'}((v,v'))=(N_G(v),v')\cup (v,N_{G'}(v')).
  \]
  Therefore, 
  \begin{align*}
    |N_{G\times G'}((v,v'))\cap (V_i,V'_j)|&=|((N_G(v),v')\cup (v,N_{G'}(v')))\cap (V_i,V'_j)|\\
    &=|((N_G(v),v')\cap (V_i,V'_j))\cup ((v,N_{G'}(v'))\cap (V_i,V'_j))|\\
    &=|(N_G(v),v')\cap (V_i,V'_j)|+|(v,N_{G'}(v'))\cap (V_i,V'_j)|. 
  \end{align*}
  Since both $P$ and $P'$ are equitable partitions, for any integers $k$ and $l$ such that $0\leq k\leq r$ and $0\leq l\leq r'$, if $(v,v')\in (V_{k},V'_{l})$, then
  \begin{align*}
    &|(N_G(v),v')\cap (V_{i},V'_{j})|=
    \begin{cases} 
    0 & j\neq l\\
    q_{ik}&j=l
    \end{cases},\\
    &|(v,N_{G'}(v'))\cap (V_i,V'_j)|=
    \begin{cases} 
    0 & i\neq k \\
    q'_{jl} & i=k
    \end{cases}.  
  \end{align*}
  Therefore,
  \begin{align*}
    |N_{G\times G'}((v,v'))\cap (V_i,V'_j)|&=|(N_G(v),v')\cap (V_i,V'_j)|+|(v,N_{G'}(v'))\cap (V_i,V'_j)|\\
     &=\delta_{ik}q'_{jl}+\delta_{jl}q_{ik},
  \end{align*}
  where $\delta_{ij}$ is the Kronecker delta defined as
  \[
    \delta_{ij}=\begin{cases}
      1 &i=j\\
      0 &i\neq j
    \end{cases}.
  \]
  This equation implies that the partition $P\times P'$ is an equitable partition. 
  The quotient matrix of $P\times P'$ is indexed by $\{0,1,\ldots,r\}\times\{0,1,\ldots,r'\}$, and its $(i,j)(k,l)$-entry 
  is denoted as $\delta_{ik}q'_{jl}+\delta_{jl}q_{ik}$. 
  We can also represent this matrix as $Q\otimes I_{r'+1}+I_{r+1}\otimes Q'$.
\end{proof}

Using \Cref{lem:cart}, we prove \Cref{thm:carteq}.

\begin{proof}[Proof of \Cref{thm:carteq}]
  Let $V,V'$ be the sets of vertices of $G$ and $G'$, respectively, and
  let $\{x_0,x_1,\ldots,x_r\}$ and $\{x'_0,x'_1,\ldots,x'_{r'}\}$ be the sets of return values of the functions $f$ and $f'$, respectively.
  Define 
  \[
    F_{x_i}(v)=\{u\in V\mid f(v,u)=x_i\}
  \] 
  as the subset of $V$ and 
  \[
    P_v=\{F_{x_0}(v)=\{v\},F_{x_1}(v),\ldots,F_{x_r}(v)\}
  \] 
  as the partition of $V$, where $v$ is for any vertex in $V$.
  Similarly, define 
  \[
    F'_{x'_i}(v')=\{u'\in V'\mid f'(v',u')=x'_i\}
  \] 
  as the subset of $V'$ and 
  \[
    P'_{v'}=\{F'_{x'_0}(v')=\{v'\},F'_{x'_1}(v'),\ldots,F'_{x'_{r'}}(v')\}
  \] 
  as the partition of $V'$, where $v'$ is for any vertex in $V'$.
  From the definition of $f$-equitable graphs and $f'$-equitable graphs, for all vertices $v\in G$ and $v'\in G'$, 
  both $P_v$ and $P'_{v'}$ are stabilized equitable partitions centered on $v$ and $v'$, respectively,
  and their quotient matrices are $Q,Q'$, respectively.
  From \Cref{lem:cart}, $P_v\times P'_{v'}$ is also a stabilized equitable partition centered on $(v,v')$ and its quotient matrix is $Q\otimes I_{r'+1}+I_{r+1}\otimes Q'$, 
  which does not depend on the choice of $v$ and $v'$.
  Here, let $(F,F')_{(i,j)}(v,v')$ be a subset of $V\times V'$ such that
  \[
    (F,F')_{(i,j)}(v,v')=\{(u,u')\in V\times V'\mid (f,f')((v,v'),(u,u'))=(i,j)\},
  \] 
  and let 
  \[
    P_{v,v'}(f,f')=\{(F,F')_{(i,j)}(v,v')\mid 0\leq i\leq r,0\leq j\leq r'\}
  \]
  be a partition of $G\times G'$.
  Since $(u,u')$ is in $(F_{x_i}(v),F'_{x'_j}(v'))$ if and only if $f(v,u)=x_i$ and $f'(v',u')=x'_j$ hold,
  we obtain $P_{v,v'}(f,f')=P_v\times P'_{v'}$.
  Therefore, $G\times G'$ is $(f,f')$-equitable with quotient matrix $Q\otimes I_{r'+1}+I_{r+1}\otimes Q'$.
\end{proof}

From \Cref{thm:carteq}, if $G$ and $G'$ are $f$-equitable and $f'$-equitable, respectively, 
and the quotient matrices of these graphs are determined, then we can also determine the average hitting time in $G\times G'$.

\subsection*{Application of \Cref{thm:carteq} for Cartesian product of two strongly regular graphs}

In this subsection, we calculate the average hitting time from any vertex to another 
on the Cartesian product of two strongly regular graphs using only their parameters by using \Cref{thm:equi} and \Cref{thm:carteq}.

First, we define a strongly regular graph.
\begin{df}
  Graph $G$ is called a \textit{strongly regular graph} if $G$ is a distance-regular graph and the diameter of $G$ is $2$.
\end{df}

Because the diameter of a strongly regular graph is always $2$, its intersection array is denoted as $\{b_0,b_1;c_1=1,c_2\}$.
Using this, the parameter of a strongly regular graph with order $n$ is generally defined as $(n,k,a,c)$, 
where $k=b_0$ is the valency of the graph, $a$ is $k-b_1-1$, and $c$ is $c_2$.

\begin{cor}\label{cor:srgcar}
  Let $G_1,G_2$ be strongly regular graphs with parameters $(n_1,k_1,a_1,c_1)$ and $(n_2,k_2,a_2,c_2)$, respectively, 
  and let the vertex sets of $G_1$ and $G_2$ be $V_1$ and $V_2$, respectively.
  Define $(d_{G_1},d_{G_2})$ as the following function:
  \begin{align*}
    (d_{G_1},d_{G_2}):(V_1\times V_2)\times(V_1\times V_2)&\mapsto \{0,1,2\}\times \{0,1,2\}\\
    (d_{G_1},d_{G_2})((u_1,u_2),(v_1,v_2))&\mapsto (d_{G_1}(u_1,v_1),d_{G_2}(u_2,v_2)).
  \end{align*}
  Then, $G_1\times G_2$ is $(d_{G_1},d_{G_2})$-equitable with the quotient matrix $Q$ as follows:
  \begin{align*}
    Q&=\scalebox{0.7}{$\left(
      \begin{array}{ccccccccc}
       0 & 1 & 0 & 1 & 0 & 0 & 0 & 0 & 0 \\
       k_2 & a_2 & c_2 & 0 & 1 & 0 & 0 & 0 & 0 \\
       0 & -a_2+k_2-1 & k_2-c_2 & 0 & 0 & 1 & 0 & 0 & 0 \\
       k_1 & 0 & 0 & a_1 & 1 & 0 & c_1 & 0 & 0 \\
       0 & k_1 & 0 & k_2 & a_1+a_2 & c_2 & 0 & c_1 & 0 \\
       0 & 0 & k_1 & 0 & -a_2+k_2-1 & a_1-c_2+k_2 & 0 & 0 & c_1 \\
       0 & 0 & 0 & -a_1+k_1-1 & 0 & 0 & k_1-c_1 & 1 & 0 \\
       0 & 0 & 0 & 0 & -a_1+k_1-1 & 0 & k_2 & a_2-c_1+k_1 & c_2 \\
       0 & 0 & 0 & 0 & 0 & -a_1+k_1-1 & 0 & -a_2+k_2-1 & -c_1-c_2+k_1+k_2 \\
      \end{array}
      \right)$}.
  \end{align*}
\end{cor}
\begin{proof}
  By \Cref{cor:drgEqui}, $G_1$ is $d_{G_1}$-equitable with the quotient matrix $Q_1$ and 
  $G_2$ is $d_{G_2}$-equitable with the quotient matrix $Q_2$, where
  \begin{align*}
    Q_1&=\begin{pmatrix}
      0 & 1 & 0 \\
      k_1 & a_1 &c_1\\
      0 & k_1-a_1-1 &k_1-c_1
    \end{pmatrix},  \\
    Q_2&=\begin{pmatrix}
      0 & 1 & 0 \\
      k_2 & a_2 &c_2\\
      0 & k_2-a_2-1 &k_2-c_2
    \end{pmatrix}.
  \end{align*}

  By \Cref{thm:carteq}, $G_1\times G_2$ is also $(d_{G_1},d_{G_2})$-equitable with the quotient matrix $Q_{1,2}$, where
  \begin{align*}
    Q_{1,2}&=Q_1\otimes I_3 + I_3\otimes Q_2\\  
    &=\scalebox{0.7}{$\left(
      \begin{array}{ccccccccc}
       0 & 1 & 0 & 1 & 0 & 0 & 0 & 0 & 0 \\
       k_2 & a_2 & c_2 & 0 & 1 & 0 & 0 & 0 & 0 \\
       0 & -a_2+k_2-1 & k_2-c_2 & 0 & 0 & 1 & 0 & 0 & 0 \\
       k_1 & 0 & 0 & a_1 & 1 & 0 & c_1 & 0 & 0 \\
       0 & k_1 & 0 & k_2 & a_1+a_2 & c_2 & 0 & c_1 & 0 \\
       0 & 0 & k_1 & 0 & -a_2+k_2-1 & a_1-c_2+k_2 & 0 & 0 & c_1 \\
       0 & 0 & 0 & -a_1+k_1-1 & 0 & 0 & k_1-c_1 & 1 & 0 \\
       0 & 0 & 0 & 0 & -a_1+k_1-1 & 0 & k_2 & a_2-c_1+k_1 & c_2 \\
       0 & 0 & 0 & 0 & 0 & -a_1+k_1-1 & 0 & -a_2+k_2-1 & -c_1-c_2+k_1+k_2 \\
      \end{array}
      \right)$}.
  \end{align*}
\end{proof}

Using \Cref{cor:srgcar}, we can calculate the average hitting times between any pair of vertices in the Cartesian product of two strongly regular graphs. 
However, when dealing with general parameters $(n_1, k_1, a_1, c_1)$ and $(n_2, k_2, a_2, c_2)$, the resulting equation becomes too long to display in full. 
Therefore, we provide specific results for the case where both graphs are identical, denoted as $G_1 = G_2 = G$.

\begin{cor}
  Let $G$ be the strongly regular graph with parameters $(n,k,a,c)$.
  Then, 
  \begin{align*}
    &H(G\times G;((u_1,v_1),(u_2,v_2)))\\
    &=\begin{cases}
      0 & (u_1,v_1)=(u_2,v_2)\\
      n^2-1 & d_{G\times G}((u_1,v_1),(u_2,v_2))=1\\
      \frac{-c \left(n^2-1\right) (a-2 k-1)+c^2 \left(n^2-1\right)+k (k+n+1)}{c (-a+c+2 k)} & (d_{G\times G}((u_1,v_1),(u_2,v_2))=2)\land(i)\\
      \frac{-a \left(n^2+n-2\right)+c \left(n^2+n-2\right)+2 k n^2+k n-4 k+n^2-n}{-a+c+2 k} & (d_{G\times G}((u_1,v_1),(u_2,v_2))=2)\land(ii)\\
      \frac{c \left(-a \left(n^2+n-2\right)+k \left(2 n^2+n-4\right)+n^2-1\right)+c^2 \left(n^2+n-2\right)+k (k+n+1)}{c (-a+c+2 k)}& d_{G\times G}((u_1,v_1),(u_2,v_2))=3\\
      \frac{c \left(-a \left(n^2+n-2\right)+k \left(2 n^2+n-4\right)+n^2+n-2\right)+c^2 \left(n^2+n-2\right)+k (2 k+n+2)}{c (-a+c+2 k)}  & d_{G\times G}((u_1,v_1),(u_2,v_2))=4,
    \end{cases}
  \end{align*}
  where condition $(i)$ is $(u_1=u_2) \lor (v_1=v_2)$ and condition $(ii)$ is otherwise.
\end{cor}
\begin{proof}
  By applying \Cref{cor:srgcar}, we know that $G\times G$ is a $(d_{G},d_{G})$-equitable graph with the quotient matrix $Q$ given as follows:
  \[
    Q=  {\scriptsize
  \left(
    \begin{array}{ccccccccc}
     0 & 1 & 0 & 1 & 0 & 0 & 0 & 0 & 0 \\
     k & a & c & 0 & 1 & 0 & 0 & 0 & 0 \\
     0 & -a+k-1 & k-c & 0 & 0 & 1 & 0 & 0 & 0 \\
     k & 0 & 0 & a & 1 & 0 & c & 0 & 0 \\
     0 & k & 0 & k & 2a & c & 0 & c & 0 \\
     0 & 0 & k & 0 & -a+k-1 & a-c+k & 0 & 0 & c \\
     0 & 0 & 0 & -a+k-1 & 0 & 0 & k-c & 1 & 0 \\
     0 & 0 & 0 & 0 & -a+k-1 & 0 & k & a-c+k & c \\
     0 & 0 & 0 & 0 & 0 & -a+k-1 & 0 & -a+k-1 & -2c+2k \\
    \end{array}
    \right)
  }.
  \]
  From \Cref{thm:equi}, 
\begin{align*}
  \vec{H}&=(h_{(0,1)},h_{(0,2)},h_{(1,0)},h_{(1,1)},h_{(1,2)},h_{(2,0)},h_{(2,1)},h_{(2,2)})\\
    &=\vec{H}(\frac{1}{2k}Q)  \\
    &=\left(\begin{array}{c}
      n^2-1 \\
      \frac{-c \left(n^2-1\right) (a-2 k-1)+c^2 \left(n^2-1\right)+k (k+n+1)}{c (-a+c+2 k)} \\
      n^2-1 \\
      \frac{-a \left(n^2+n-2\right)+c \left(n^2+n-2\right)+2 k n^2+k n-4 k+n^2-n}{-a+c+2 k} \\
      \frac{c \left(-a \left(n^2+n-2\right)+k \left(2 n^2+n-4\right)+n^2-1\right)+c^2 \left(n^2+n-2\right)+k (k+n+1)}{c (-a+c+2 k)}\\
      \frac{-c \left(n^2-1\right) (a-2 k-1)+c^2 \left(n^2-1\right)+k (k+n+1)}{c (-a+c+2 k)}\\
      \frac{c \left(-a \left(n^2+n-2\right)+k \left(2 n^2+n-4\right)+n^2-1\right)+c^2 \left(n^2+n-2\right)+k (k+n+1)}{c (-a+c+2 k)}\\
      \frac{c \left(-a \left(n^2+n-2\right)+k \left(2 n^2+n-4\right)+n^2+n-2\right)+c^2 \left(n^2+n-2\right)+k (2 k+n+2)}{c (-a+c+2 k)}  \\
    \end{array}\right)^{\top}.
\end{align*}
  Then, 
  \[
  H(G\times G;((u_1,v_1),(u_2,v_2)))=\begin{cases}
    0 & (u_1,v_1)=(u_2,v_2)\\
    h_{(d_G(u_1,u_2),d_G(v_1,v_2))} &(u_1,v_1)\neq(u_2,v_2)
  \end{cases}.\]
  By organizing the conditions and combining equal values, the desired equation can be obtained.
\end{proof}

\begin{rem}
  Using \Cref{thm:carteq}, we can also determine the average hitting times in the Cartesian product of two distance-regular graphs
  by their intersection arrays.
  However, as the number of parameters increases, the equation becomes more complex. 
  Therefore, it is not practical to represent the average hitting times by their parameters in the Cartesian product of two typical distance-regular graphs. 
  On the other hand, it may be possible to provide the bounds of the average hitting times in such a graph.
  The bounds of the average hitting times in a distance-regular graph were given by Koolen and Markowsky~\cite{KOOLEN2016737},
  but in the Cartesian product of two typical distance-regular graphs, such bounds are not known.   
\end{rem}

\section{Average hitting time of generalized Paley graph in two specific cases}\label{sec:genPaley}

In this section, first we give a definition of a generalized Paley graph 
and determine a function $f_{\alpha,k}$ such that a generalized Paley graph becomes an $f_{\alpha,k}$-equitable graph.
We also introduce the concept of a class of graphs called quasi-strongly regular graphs and define the parameters of a generalized Paley graph.
Following this, we compute the quotient matrix of two specific generalized Paley graph as $f_{\alpha,k}$-equitable graphs, using only these parameters.
We then proceed to calculate their average hitting times. 

\subsection*{Definition of generalized Paley graph and some of its properties}

Let $\FF_x$ be a finite field with order $x$, and let $\FF_x^*$ be the multiplicative group of $\FF_x$. 
Define the $k$th power of $\FF_x^*$ as $(\FF_x^*)^k =\{n^k \mid n \in \FF_x^*\}$. 

\begin{df}[Generalized Paley graph~\cite{10.1307/mmj/1242071694}]
  Let $x$ be a prime power, and let $k$ be an integer such that $x-1 \equiv 1 \pmod{2k}$. 
  The generalized Paley graph, denoted as $GP(x,k)$, is defined with the vertex set $\FF_x$, 
  and for any vertices $n$ and $m$ in $GP(x,k)$, $n \sim m$ if $n - m \in (\FF_x^*)^k$.   
\end{df}

Then, we obtain the following theorem.

\begin{thm}\label{thm:genEqui}
  Let $\alpha$ be the generator of $\FF_x^*$, and let $f_{\alpha,k}$ be a function such that 
  \begin{align*}
    f_{\alpha,k}:\FF_x\times\FF_x &\mapsto \{0,1,\ldots,k\}\\
    f_{\alpha,k}(u,v)&\mapsto 
    \begin{cases}
      i & v-u\in \alpha^{i-1}(\FF_x^*)^k\\
      0 & v-u=0
    \end{cases}.
  \end{align*}
  Then, $GP(x,k)$ is an $f_{\alpha,k}$-equitable graph.
\end{thm}

\begin{proof}
  For any $l\in(\FF_x^*)^k$, we can consider an automorphism of $\FF_x$, denoted as $\tau_{l,\times}$, such that $\tau_{l,\times}(n)=ln$.
  Since $(\FF_x^*)^k$ is a multiplicative subgroup of $\FF_x^*$, for any vertices $u$ and $v$ in $GP(x,k)$,
\begin{align*}
  u\sim v &\Longleftrightarrow u-v\in(\FF_x^*)^k\\
  &\Longleftrightarrow \tau_{l,\times}(u-v)\in(\FF_x^*)^k\\
  &\Longleftrightarrow \tau_{l,\times}(u)\sim \tau_{l,\times}(v).
\end{align*}
Therefore, $\tau_{l,\times}$ is also an automorphism of $GP(x,k)$.
  Let \[\Pi=\{\tau_{l,\times}\mid l\in(\FF_x^*)^k\}\] be the automorphism set. 
  Since $\tau_{l,\times}(0)=0$ for any $l\in(\FF_x^*)^k$, $\tau_{l,\times}$ is a subset of the stabilizer subgroup of vertex $0$. 
  Define $P_0$ to be a partition based on the orbits from $\Pi$. 
  Then, we can denote $P_0$ as follows:
  \[
    P_0=\{V_0=\{0\},V_1=(\FF_x^*)^k,V_2=\alpha(\FF_x^*)^k,V_3=\alpha^2(\FF_x^*)^k,\ldots,V_{k}=\alpha^{k-1}(\FF_x^*)^k\}.
  \]
  
  For any integers $i$ and $j$ such that $0\leq i,j\leq k$ and for any $v$ and $v'$ in $V_i$, 
  there exists an automorphism $\tau_{\times,l}$ such that $\tau_{\times,l}(v)=v'$ and $\tau_{\times,l}(V_j)=V_j$.
  Therefore, 
  \[
  |N(v)\cap V_j|=|N(\tau_{\times,l}(v))\cap \tau_{\times,l}(V_j)|=|N(v')\cap V_j|.
  \]
  This shows that $P_0$ is a stabilized equitable partition of $GP(x,k)$ centered on $0$. Here, let $F_{i}(n)=\{v\in \FF_x\mid f_{\alpha,k}(n,v)=i\}$, where $n$ is for any element of $\FF_x$. 
  Then, from the definition of $f_{\alpha,k}$, we obtain that $F_i(0)=\alpha^{i-1}(\FF_x^*)^k$. 
  Therefore, we can also denote $P_0$ by using $F_i(0)$ as follows:
  \[
    P_0=\{F_0(0)=\{0\},F_1(0),\ldots,F_k(0)\}.
  \]
  Next, we prove that 
  \[
    P_n=\{F_0(n)=\{n\},F_1(n),\ldots,F_k(n)\}
  \]
  is a stabilized equitable partition of $GP(x,k)$ centered on $n$.
  Let $\sigma_{n,+}$ be a permutation of $\FF_x$ such that $\sigma_{n,+}(a)=a+n$, where $n$ is for any element of $\FF_x$.
  From the definition of $GP(x,k)$, for any vertices $u$ and $v$ in $GP(x,k)$,
\begin{align*}
  u\sim v &\Longleftrightarrow u-v\in(\FF_x^*)^k\\
  &\Longleftrightarrow \sigma_{n,+}(u)-\sigma_{n,+}(v)\in(\FF_x^*)^k\\
  &\Longleftrightarrow \sigma_{n,+}(u)\sim \sigma_{n,+}(v).
\end{align*}
  Therefore, $\sigma_{n,+}$ is a graph automorphism of $GP(x,k)$.
  Since $P_0$ is an equitable partition,
  \[
    \sigma_{n,+}(P_0)=\{\sigma_{n,+}(F_0(0))=\{n\},\sigma_{n,+}(F_1(0)),\ldots,\sigma_{n,+}(F_k(0))\}
  \]
  is also a stabilized equitable partition centered on $n$.
  Additionally,
  \begin{align*}
    v\in F_i(0)&\Longleftrightarrow f(0,v)=i\\
    &\Longleftrightarrow f(n,n+v)=i\\
    &\Longleftrightarrow \sigma_{n,+}(v)\in F_i(n).
  \end{align*}
  Therefore,
  \[
    \sigma_{n,+}(P_0)=\{F_0(n),F_1(n),\ldots,F_k(n)\}=P_n.
  \]
  Hence, for every vertex $n$, $P_n$ is a stabilized equitable partition of $GP(x,k)$ centered on $n$.
  Finally, we check the condition of the quotient matrix.
  Let the quotient matrix of $P_n$ be $q(n)_{ij}$.
  From the definition of the quotient matrix,
  \[
  q(n)_{ij}=|N(v)\cap F_i(n)|,
  \]
  where $v\in F_j(n)$.
  Since $\sigma_{n,+}(F_i(0))=F_i(n)$, we obtain 
  \begin{align*}
    q(n)_{ij}&=|N(v)\cap F_i(n)|\\
    &=|N(\sigma_{n,+}(v-n))\cap \sigma_{n,+}(F_i(0))|\\
    &=|N(v-n)\cap F_i(0)|=q(0)_{ij}.
  \end{align*}
  This implies that the quotient matrix of $P_n$ does not depend on the choice of $n$.
  Therefore, $GP(x,k)$ is an $f_{\alpha,k}$-equitable graph.
\end{proof}

The graph $GP(x,k)$ also belongs to the category of graphs known as quasi-strongly regular.
\begin{df}[Quasi-strongly regular~\cite{golightly1997family}]
  Let $G$ be a $k$-regular graph with order $n$. 
  $G$ is called quasi-strongly regular with parameter $(n,k,c_1;\{c_2,c_3,\ldots,c_l\})$ if $G$ satisfies the following for all pairs of vertices $u,v\in G$: 
  \[
    |N(u)\cap N(v)|\in\begin{cases}
      \{c_1\} & u\sim v,\\
      \{c_2,c_3,\ldots,c_l\} &u\nsim v.
    \end{cases}
  \]
\end{df}

When considering $GP(x,k)$, we can define $F_{i}(n)$ as the set $\{v \in \FF_x \mid f_{\alpha,k}(n,v) = i\}$ and create partitions for each vertex $n$ as follows:
\begin{align*}
  P_n&=\{F_0(n)=\{n\},F_1(n),\ldots,F_k(n)\}.
\end{align*}
Since 
\[
F_1(n)=\{u\in V\mid u-n\in (\FF_x^*)^k\} =N(n)
\]
and $GP(x,k)$ is $f_{\alpha,k}$-equitable, for every vertex $v$ in $N(n)$, $|N(v) \cap N(n)|$ does not depend on the choice of vertex $v$.
Therefore, $GP(x,k)$ is quasi-strongly regular. Let $c_i = |N(0) \cap N(v)|$, where $v \in F_i(0)$. 
Then, the parameters of $GP(x,k)$ as a quasi-strongly regular graph are $(x, \frac{x-1}{k}, c_1; \{c_2, c_3, \ldots, c_k\})$.

\subsection*{Some elements of quotient matrix of $GP(x,k)$ as $f_{\alpha,k}$-equitable graph}

From \Cref{thm:equi} and \Cref{thm:genEqui}, it is sufficient to know the quotient matrix of $P_0$ to compute the average hitting time of $GP(x,k)$. 
To compute it, we first determine some elements of the quotient matrix of $P_0$ by the parameters of $GP(x,k)$ as quasi-strongly regular.
In the following, we denote $P(x,k)=\{F_0,F_1,\ldots,F_k\}$ as the partition $P_0$ of $GP(x,k)$ and $q(x,k)_{ij}$ as the quotient matrix of $P_0$.

\begin{lem}\label{lem:quo1}
  \begin{align*}
    q(x,k)_{0,i}&=
    \begin{cases}
      1 & i=1\\
      0 & i\neq 1  
    \end{cases}\\
    q(x,k)_{i,0}&=
    \begin{cases}
      \frac{x-1}{k} & i=1\\
      0 & i\neq 1  
    \end{cases}\\
    q(x,k)_{1,i}&=c_i
  \end{align*}
  Additionally, for any integers $i,j$ such that $i,j\geq 1$,
  \begin{align*}
    q_{ij}&=q_{ji}.      
  \end{align*}
\end{lem}

\begin{proof}
Because $N(0)=(\FF_x^*)^k=F_1$, we obtain $q(x,k)_{0,1}=1,q(x,k)_{1,0}=\frac{x-1}{k}$ and for any integer $i$ such that $2\leq i\leq k$, 
$q(x,k)_{0,i}=q(x,k)_{i,0}=0$.
From the definition of $c_i$, for any $1\leq i\leq k$,
\begin{equation*}\label{eq:quo1}
  v\in F_i\Longrightarrow c_i=|N(0)\cap N(v)|=|F_1\cap N(v)|=q(x,k)_{1,i}.  
\end{equation*}

From $|F_1|=|F_2|=\ldots=|F_k|=\frac{x-1}{k}$ and \Cref{eq:edge}, we have 
\begin{equation*}\label{eq:quo2}
  q_{ij}=q_{ji}.    
\end{equation*}
\end{proof}

We can also determine $q(x,k)_{i,i}$.

\begin{lem}\label{lem:quo2}
  \[
    q(x,k)_{i,i}=c_{k-i+2} 
  \]
\end{lem}

\begin{proof}  
From an algebraic perspective, $|N(v)\cap F_i|$ is the number of pairs of solutions $(a_1^k,\alpha^{i-1}a_2^k)$ to the following polynomial with coefficients in the finite field $\FF_x$, neither of which is $0$:
\[
v-a_1^k=\alpha^{i-1}a_2^k.
\] 
We denote $S(g(a_1,a_2,\ldots,a_l))$ as the number of sets of solutions $(a_1,a_2,\ldots,a_l)$ to equation $g(a_1,a_2,\ldots,a_l)=0$, with no $a_i$ equal to $0$.
Then, when $v\in F_j$, 
\begin{align*}
  q(x,k)_{ij}=|N(v)\cap F_i|&=\frac{|S(a_1^k+\alpha^{i-1}a_2^k-v)|}{k^2},
\end{align*}
so
\[
  |S(a_1^k+\alpha^{i-1}a_2^k-v)|=k^2q(x,k)_{ij}.  
\]
From this equation, we obtain that $|S(a_1^k+\alpha^{i-1}a_2^k-v)|$ does not depend on the choice of $v\in F_j$.
Here, for any $v\in F_j$, there exists $m\in \FF_x^*$ such that $v=\alpha^{j-1}m^k$.
Therefore,
\begin{align}\label{eq:quoSol}
  |S(a_1^k+\alpha^{i-1}a_2^k-\alpha^{j-1}a_3^k)|&=\sum_{m=1}^{x-1}|S(a_1^k+\alpha^{i-1}a_2^k-\alpha^{j-1}m^k)|\nonumber\\
  &=(x-1)|S(a_1^k+\alpha^{i-1}a_2^k-v)|\nonumber\\
  &=(x-1)k^2q(x,k)_{ij}.  
\end{align}

From \Cref{lem:quo1}, if $j=1$ then
\begin{equation}\label{eq:quo5}
  |S(a_1^k+\alpha^{i-1}a_2^k-a_3^k)|=(x-1)k^2c_i.
\end{equation}
Here, 
\begin{align*}
  a_1^k+\alpha^{i-1}a_2^k-\alpha^{j-1}a_3^k=0 \Longleftrightarrow \alpha^la_1^k+\alpha^{i-1+l}a_2^k-\alpha^{j-1+l}a_3^k=0.
\end{align*}
Therefore, 
\begin{equation}\label{eq:quo4}
  |S(a_1^k+\alpha^{i-1}a_2^k-\alpha^{j-1}a_3^k)|=|S(\alpha^{l}a_1^k+\alpha^{i-1+l}a_2^k-\alpha^{j-1+l}a_3^k)|.
\end{equation}
From \Cref{eq:quoSol,eq:quo5,eq:quo4},
\begin{align*}
  (x-1)k^2q(x,k)_{i,i}&=|S(a_1^k+\alpha^{i-1}a_2^k-\alpha^{i-1}a_3^k)|\\
  &=|S(\alpha^{k-(i-1)}a_1^k+\alpha^{k}a_2^k-\alpha^{k}a_3^k)|\\
  &=|S(\alpha^{k-(i-1)}a_1^k+a_2^k-a_3^k)|\\
  &=|S(a_1^k+\alpha^{k-(i-1)}a_2^k-a_3^k)|\\
  &=(x-1)k^2c_{k-i+2}.
\end{align*}
Hence,
\begin{equation*}
  q(x,k)_{i,i}=c_{k-i+2}.
\end{equation*}
\end{proof}

Especially, we can compute all of the entries of $q(x,3)_{ij}$ and $q(x,4)_{ij}$ from their parameters.

\subsection*{Average hitting time of $GP(x,3)$}

In this subsection, we compute the average hitting time of $GP(x,3)$ by its parameter as a quasi-strongly regular graph.

\begin{thm}\label{thm:gen3}
  Let $\alpha$ be the generator of $\FF_x^*$, and let $c_i=|N(v)\cap N(u)|$, where $u-v\in \alpha^{i-1}(\FF_x^*)^3$ and $i$ is for any integer such that $1\leq i\leq 3$.
  \[
    H(GP(x,3);(v,u))=\begin{cases}
      0 & u-v=0\\
      x-1 & u-v\in (\FF_x^*)^3\\
      \frac{\frac{x-1}{3} (c_1 (3 c_2+3 c_3+2)+3 c_2 (c_3+1)+4 c_3+2)}{c_2 (c_1+c_3+1)+(c_1+1) c_3}& u-v\in\alpha(\FF_x^*)^k\\
      \frac{\frac{x-1}{3}(c_1 (3 c_2+3 c_3+2)+c_2 (3 c_3+4)+3 c_3+2)}{c_2 (c_1+c_3+1)+(c_1+1) c_3}& u-v\in\alpha^2(\FF_x^*)^k
    \end{cases}  
    \]
\end{thm}
\begin{proof}
From \Cref{lem:quo1} and \Cref{lem:quo2}, there are only two entries that have not been computed, namely, $q(x,3)_{2,3}$ and $q(x,3)_{3,2}$.
However, from \Cref{eq:deg} and the fact that the valency of $GP(x,3)$ is $\frac{x-1}{3}$,
\[
\sum_{i=0}^3 q(x,3)_{i,2}=\sum_{i=0}^3 q(x,3)_{i,3}  =\sum_{i=0}^3 q(x,1)_{i,3}=1+c_1+c_2+c_3.
\]
Therefore, 
\[
q(x,3)_{ij}=
\begin{pmatrix}
  0&1&0&0\\
  \frac{x-1}{3}&c_1&c_2&c_3\\
  0&c_2&c_3&1+c_1\\
  0&c_3&1+c_1&c_2  
\end{pmatrix}.  
\]
From \Cref{thm:genEqui}, $GP(x,3)$ is an $f_{\alpha,3}$-equitable graph.
We define $(h_1,h_2,h_3)$ as follows:
\begin{align*}
  (h_1,h_2,h_3)&=\vec{H}(q(x,3)_{ij})\\
  &=\begin{pmatrix}
    x-1\\
    \frac{\frac{x-1}{3} (c_1 (3 c_2+3 c_3+2)+3 c_2 (c_3+1)+4 c_3+2)}{c_2 (c_1+c_3+1)+(c_1+1) c_3}\\
    \frac{\frac{x-1}{3}(c_1 (3 c_2+3 c_3+2)+c_2 (3 c_3+4)+3 c_3+2)}{c_2 (c_1+c_3+1)+(c_1+1) c_3}
  \end{pmatrix}^{\top}.
\end{align*} 
  From \Cref{thm:equi}, since $GP(x,3)$ is $f_{\alpha,3}$-equitable,
  \begin{align*}
    H(GP(x,3);(v,u))&=\begin{cases}
      0 & u=v\\
      h_{f_{\alpha,k}(u,v)} &u\neq v
    \end{cases}.
  \end{align*}
  By organizing the conditions, the desired equation can be obtained.
\end{proof}

\subsection*{Average hitting time of $GP(x,4)$}

In this subsection, we compute the average hitting time of $GP(x,3)$ 
by its parameter as a quasi-strongly regular graph.

\begin{thm}\label{thm:gen4}
  Let $\alpha$ be the generator of $\FF_x^*$, let $c_i=|N(v)\cap N(u)|$, where $u-v\in \alpha^{i-1}(\FF_x^*)^4$ and, for any integer $i$ such that $1\leq i\leq 4$, 
  define $h_i$ as follows:
    \begin{align*}
      h_1&= x-1\\
      h_2&= \frac{(x-1)\left(
        \begin{array}{c}
           1280 c_3^3-64 c_3^2 (8 c_4+7 x-13)\\
           +4 c_3 \left(64 c_4^2+32 c_4 (x-2)+9 x^2-48 x+39\right)\\
           -(x-1) \left(64 c_4^2-24 c_4-9 x+9\right)
        \end{array}
      \right)
        }{64 \left(
        \begin{array}{c}
          20 c_3^3+c_3^2 (-8 c_4-7 x+7)+4 c_3 c_4^2\\
          +2 c_3 c_4 (x-1)+\frac{9}{16} c_3 (x-1)^2-c_4^2 (x-1)  
        \end{array}        
          \right)}\\
      h_3&= \frac{(x-1) \left(
      \begin{array}{c}
        1280 c_3^3-64 c_3^2 (8 c_4+7 x-10)\\
        +4 c_3 \left(64 c_4^2+32 c_4 x+9 x^2-42 x+33\right)\\
        -2 \left(32 c_4^2+9\right) x+9 x^2+9  
      \end{array}  
      \right)}{64 \left(
      \begin{array}{c}
        20 c_3^3+c_3^2 (-8 c_4-7 x+7)+4 c_3 c_4^2\\
        +2 c_3 c_4 (x-1)+\frac{9}{16} c_3 (x-1)^2-c_4^2 (x-1)  
      \end{array}  
        \right)}\\
      h_4&= \frac{(x-1) \left(
        \begin{array}{c}
          1280 c_3^3-64 c_3^2 (8 c_4+7 x-9)\\
          +4 c_3 \left(64 c_4^2+32 c_4 x+9 \left(x^2-4 x+3\right)\right)\\
          -(x-1) \left(64 c_4^2+24 c_4-9 x+9\right)  
        \end{array}
      \right)}{64 \left(
      \begin{array}{c}
        20 c_3^3+c_3^2 (-8 c_4-7 x+7)+4 c_3 c_4^2\\
        +2 c_3 c_4 (x-1)+\frac{9}{16} c_3 (x-1)^2-c_4^2 (x-1)  
      \end{array}  
        \right)}.
    \end{align*}
  Then, 
  \[
    H(GP(x,4);(v,u))=\begin{cases}
      0 & u-v=0\\
      h_1 & u-v\in (\FF_x^*)^4\\
      h_2& u-v\in \alpha(\FF_x^*)^4\\
      h_3& u-v\in \alpha^2(\FF_x^*)^4\\
      h_4& u-v\in \alpha^3(\FF_x^*)^4\\
    \end{cases}.
    \]    
\end{thm}

To prove \Cref{thm:gen4}, we should know all the entries of $Q(x,4)$.
Therefore, we prove the following lemma.

\begin{lem}\label{lem:quoFor4}
  \[
    Q(x,4)=
    \begin{pmatrix}
      0&1&0&0&0\\
      \frac{x-1}{4}&c_1&c_2&c_3&c_4\\
      0&c_2&c_4&\frac{1+c_1+c_3}{2}&\frac{1+c_1+c_3}{2}\\
      0&c_3&\frac{1+c_1+c_3}{2}&c_3&\frac{1+c_1+c_3}{2}\\
      0&c_4&\frac{1+c_1+c_3}{2}&\frac{1+c_1+c_3}{2}&c_2
    \end{pmatrix}  
    \]    
\end{lem}
\begin{proof}
  $x-1\equiv 0 \pmod{2k}$, $-1$ is in $(\FF_x^*)^k$ and there exists the element $\beta$ in $\FF_x$ such that $\beta^k=-1$.
  Then, 
  \[
    a_1^k+\alpha^{i-1}a_2^k-\alpha^{j-1}a_3^k=0 \Longleftrightarrow -(\beta a_1)^k+\alpha^{i-1}a_2^k-\alpha^{j-1}a_3^k=0.
  \]
  Therefore, we obtain 
  \[
    |S(a_1^k+\alpha^{i-1}a_2^k-\alpha^{j-1}a_3^k)|=|S(-a_1^k+\alpha^{i-1}a_2^k-\alpha^{j-1}a_3^k)|.
  \]
  Similar reasoning can be applied to the coefficients of $a_2$ and $a_3$ as well.
  Hence, the signs of the terms in $xa_1^k+ya_2^k+za_3^k$ do not affect $|S(xa_1^k+ya_2^k+za_3^k)|$.
  In the case $k=4$, from \Cref{eq:quoSol,eq:quo4},
  \begin{align*}
    (x-1)4^2q(x,4)_{4,2}&=|S(a_1^4+\alpha a_2^4-\alpha^3a_3^4)|\\
    &=|S(\alpha a_1^4+\alpha^2 a_2^4-(\alpha a_3)^4)|\\
    &=|S(a_1^4+\alpha a_2^4-\alpha^2a_3^4)|=(x-1)4^2q(x,4)_{3,2}.
  \end{align*}  
  Therefore, $q(x,4)_{3,2}=q(x,4)_{4,2}$.
  From \Cref{eq:deg},
  \[
  \sum_{i=0}^{4}  q(x,4)_{i,2}=c_2+c_4+q(x,4)_{3,2}+q(x,4)_{4,2}=1+c_1+c_2+c_3+c_4.
  \]
  Hence, $q(x,4)_{3,2}=q(x,4)_{4,2}=\frac{1+c_1+c_3}{2}$.
  From \Cref{lem:quo1} and \Cref{lem:quo2}, we can compute the remaining entries:
  \[
  Q(x,4)=
  \begin{pmatrix}
    0&1&0&0&0\\
    \frac{x-1}{k}&c_1&c_2&c_3&c_4\\
    0&c_2&c_4&\frac{1+c_1+c_3}{2}&\frac{1+c_1+c_3}{2}\\
    0&c_3&\frac{1+c_1+c_3}{2}&c_3&\frac{1+c_1+c_3}{2}\\
    0&c_4&\frac{1+c_1+c_3}{2}&\frac{1+c_1+c_3}{2}&c_2
  \end{pmatrix}.  
  \]    
\end{proof}

Because the column sums of $Q(x,4)$ are always $\frac{x-1}{4}$, 
there are non-trivial constraints on the parameters.

\begin{cor}\label{cor:par}
  \[
  c_2+c_4=2c_3.  
  \]
\end{cor}

\begin{proof}
  \begin{align*}
    \sum_{i=0}^{r}q(x,4)_{i,1}&=\sum_{i=0}^{r}q(x,4)_{i,3}\\
    \Longleftrightarrow 1+c_1+c_2+c_3+c_4&=1+c_1+3c_3\\
    \Longleftrightarrow c_2+c_4&=2c_3.  
  \end{align*}
\end{proof}

Next, we give a proof of \Cref{thm:gen4}.

\begin{proof}[Proof of \Cref{thm:gen4}]
  From \Cref{thm:genEqui,lem:quoFor4}, $GP(x,4)$ is $f$-equitable with the following quotient matrix $q(x,4)_{ij}$:
  \[
  q(x,4)_{ij}=
  \begin{pmatrix}
    0&1&0&0&0\\
    \frac{x-1}{k}&c_1&c_2&c_3&c_4\\
    0&c_2&c_4&\frac{1+c_1+c_3}{2}&\frac{1+c_1+c_3}{2}\\
    0&c_3&\frac{1+c_1+c_3}{2}&c_3&\frac{1+c_1+c_3}{2}\\
    0&c_4&\frac{1+c_1+c_3}{2}&\frac{1+c_1+c_3}{2}&c_2
  \end{pmatrix}.
  \] 
  By computing $\vec{H}(q(x,4)_{ij})$, we have 
\begin{align*}
  (h_1,h_2,h_3,h_4)=\vec{H}(q(x,4)_{ij}).
\end{align*} 
  From \Cref{thm:equi}, because  $GP(x,4)$ is $f_{\alpha,4}$-equitable,
  \begin{align*}
    H(G;\{u,v\})&=\begin{cases}
      0 & u=v\\
      h_{f(u,v)} &u\neq v
    \end{cases}\\
    &=\begin{cases}
        0 & u-v=0\\
        h_1 & u-v\in (\FF_x^*)^4\\
        h_2& u-v\in \alpha(\FF_x^*)^4\\
        h_3& u-v\in \alpha^2(\FF_x^*)^4\\
        h_4& u-v\in \alpha^3(\FF_x^*)^4\\
      \end{cases}.     
  \end{align*}
  
\end{proof}

\begin{rem}
  We determined $P(x,4)$ and obtained \Cref{cor:par} as a side effect.
  From an algebraic perspective, the parameters of $GP(x,k)$ as a quasi-strongly regular graph can be calculated 
  using the \textit{Jacobi sum}, which is a type of character sum.
  Therefore, we may obtain relational equations for some other $k$ similar to \Cref{cor:par} using the properties of the Jacobi sum.
  If such equations are found for some other $k$, we may calculate $P(x,k)$ using only the parameters of $GP(x,k)$ as a quasi-strongly regular graph,
  and also determine the average hitting times in $GP(x,k)$.
\end{rem}

\begin{rem}
  We proved that $G$ is a distance-regular graph if and only if $G$ is a $d_G$-equitable graph.
  We also claimed that the Cartesian product of two strongly regular graphs, denoted as $G_1$ and $G_2$, is a $(d_{G_1},d_{G_2})$-equitable graph and 
  $GP(x,k)$ is an $f_{\alpha,k}$-equitable graph. 
  In these cases, $(d_{G_1},d_{G_2})$ and $f_{\alpha,k}$ are ``limited'' functions, which means that they can be defined for only specific graphs.
  On the other hand, the function $d_G$ can be defined for any graph.
  We consider that if a function $f$ can be considered on general graphs, then 
  $f$-equitable graphs may have some properties similar to distance-regular graphs,
  such as the bounds of the average hitting times, which is the result of Koolen and Markowsky~\cite{KOOLEN2016737}.

To give an example of such a function, we can consider the function $n_G(u,v)$ that represents the number of common neighbors of two vertices $u$ and $v$ in graph $G$.
  We define $G$ as being neighbor-equitable if $G$ is an $n_G$-equitable graph.
  The simplest example of neighbor-equitable is a strongly regular graph with parameters $(n,d,a,c)$, where $d,a$ and $c$ are all distinct integers.
  From this, although there are some constraints on the parameters,  neighbor-equitable graphs can also be considered as a generalization of strongly regular graphs.
  
  Let the parameters of $GP(x,k)$ as quasi-strongly regular be $(x,\frac{x-1}{k},c_1;c_2,\ldots,c_k)$.
  In fact, when $c_1,c_2,\ldots,c_k$ are all distinct integers, then $GP(x,k)$ is also neighbor-equitable.
  This is because $F_i(n)=N_{c_i}(n)$ holds for all integers $i$ and $n$ such that $0\leq i\leq k$ and $n\in \FF_x$,
  where $N_{c_i}(n)=\{v\in GP(x,k)\mid |N(n)\cap N(v)|=c_i\}$.
  It is not known what other types of neighbor-equitable graphs may exist.
  
\end{rem}

\section*{Acknowledgments}

The author thanks Professors Tsuyoshi~Miezaki, Akihiro~Munemasa, and Hiroshi~Suzuki and Ph.D. student Yuuho~Tanaka for their helpful discussions and comments.



\end{document}